\documentclass{amsart}[12pt]
\usepackage{graphicx}
\usepackage{psfrag}
\allowdisplaybreaks[1] 

\usepackage{latexsym}
\usepackage{amssymb}
\usepackage{amsmath}
\usepackage{amscd}
\usepackage{amsthm}

\newcommand{\R}{\mathbb{R}}

\newcommand{\la}{\left\langle}
\newcommand{\ra}{\right\rangle}

\newtheorem{thm}{Theorem}[section]
\newtheorem*{thm*}{Theorem}
\newtheorem{cor}[thm]{Corollary}
\newtheorem{lemma}[thm]{Lemma}
\newtheorem{prop}[thm]{Proposition}
\newtheorem*{prop*}{Proposition}

\numberwithin{equation}{section}
\newtheorem{main}{Theorem}
\newtheorem{quest}{Question}
\newtheorem{corM}[main]{Corollary}

\theoremstyle{definition}

\numberwithin{equation}{section}

\title{Scattering Boundary Rigidity in the Presence of a Magnetic Field}

\author{Pilar Herreros}
\address{Mathematisches Institut, University of M\"unster, 48149 M\"unster, Germany}
\email{p.herreros@uni-muenster.de}
\begin{document}

\begin{abstract}
It has been shown in \cite{DPSU} that, under some additional assumptions, two simple domains with the same scattering data are equivalent. We show that the simplicity of a region can be read from the metric in the boundary and the scattering data. This lets us extend the results in \cite{DPSU} to regions with the same scattering data, where only one is known apriori to be simple. We will then use this results to resolve a local version of a question by Robert Bryant. That is, we show that a surface of constant curvature can not be modified in a small region while keeping all the curves of some fixed constant geodesic curvatures closed.
\end{abstract}

\maketitle

\section{Introduction}

A magnetic field on a Riemannian manifold can be represented by a closed
$2-$form $\Omega$, or equivalently by the $(1,1)$ tensor $Y:TM\to TM$ defined
by $\Omega(\xi,\nu)=\langle Y(\xi), \nu\rangle$ for all $x\in M$ and $\xi,
\nu\in T_xM$. The trajectory of a charged particle in such a magnetic field is
then modeled by the equation $$\nabla_{\gamma'}\gamma'=Y(\gamma'),$$ we will
call such curves \emph{magnetic geodesics}. In contrast to regular (or \emph{straight}) geodesics,
magnetic geodesics are not reversible, and can't be rescaled, i.e. the
trajectory depends on the energy $|\gamma'|^2$.\\

Magnetic geodesics and the magnetic flow where first considered by V.I. Arnold
\cite{A61} and D.V. Anosov and Y.G. Sinai \cite{AS}. The existence of closed
magnetic geodesics, and the magnetic flow in general, has been widely studied
since then. Some of the approaches to this subject are, the Morse-Novikov
theory for variational functionals (e.g. \cite{No, NT, Ta}), Aubry-Mather's
theory (e.g. \cite{CMP}),  the theory of dynamical systems (e.g. \cite{Gr99,
Ni, PP97})  and  using methods from symplectic geometry (e.g. \cite{A86, Gi87,
Gi96}).\\

In the special case of surfaces, where the $2-$form has the form $\Omega=k(x)
dA$, a magnetic geodesic of energy $c$ has geodesic curvature $k_g=
k(x)/\sqrt{c}$. This relates magnetic geodesics with the problem of prescribing
geodesic curvature, in particular with the study of curves of constant geodesic
curvature. This relation was used by V.I. Arnold in \cite{A88}, and later by
many others (see e.g \cite{Le, Sch}), to study the existence of closed
curves with prescribed geodesic curvature.

It is clear that on surfaces of constant curvature the curves of large constant
geodesic curvature are circles, therefore closed. The study of these curves
goes back to Darboux, who in 1894 claimed (in a footnote in his book \cite{Da})
that the converse is true, that is, if all curves of constant (sufficiently
large) geodesic curvature are closed, then the surface has to be of constant
Gauss curvature.

The proof of this result depends strongly on the fact that curves of geodesic
curvature are closed for all large curvature, or equivalently low energy. This
raises the following question, brought to my attention by R. Bryant.

\begin{quest}\label{BryantQ}
Are surfaces of constant Gauss curvature the only surfaces for which all curves
of a fixed constant nonzero geodesic curvature are closed?

\end{quest}

For the case where the constant is $0$, this question corresponds to existence
of surfaces all of whose geodesics are closed. The first examples of such
surfaces where given by Zoll \cite{Zo} who, in 1903, constructed a surface of
revolution with this property.\\

Using the relation between geodesic curvature and magnetic geodesics we can
approach this question by studying magnetic geodesics on a surface, in the
presence of a constant magnetic field. A first step in this direction is to
determine if a surface of constant curvature can be locally changed keeping all
the curves of a fixed constant geodesic curvature closed. We show in section
\ref{Rigidity for Surfaces} that the metric can not be changed in a small
region without loosing this property. In fact, more generally, we show that a
Riemannian surface with a magnetic flow whose orbits are closed can't be
changed locally
without ``opening'' some of its orbits. \\

The easier way of changing the metric in a small region without opening the
orbits is to require that all magnetic geodesics that enter the region leave it
at the same place and in the same direction as before, to join the outside part
of the orbit. This is the \emph{magnetic scattering data} of the region; for
each point and inward direction on the boundary, it associates the exit point
and direction of the corresponding unit speed magnetic geodesic.\\

With this problem in mind, we can ask the following boundary rigidity question
for magnetic geodesics. In a Riemannian manifold with boundary, in the presence
of a magnetic field, is the metric determined by the metric on the boundary and
the magnetic scattering data?

In general this is not true, even for the geodesic case. For example, a round
sphere with a small disk removed has the same scattering data as a round $\R
P^2$ with a disk of the same size removed. One of the usual conditions to
obtain boundary rigidity is to assume that the region is \emph{simple}. In our
setting \emph{simple} means a compact region that is magnetically convex, and
where the magnetic exponential map has no conjugate points (see section
\ref{Simple metrics and boundary data}).

For simple domains, scattering rigidity for geodesics is equivalent to distance
boundary rigidity (see \cite{Cr04}) and it has been widely studied. It is known
to hold for simple subdomains of $\R^n$ \cite{Gro} or \cite{Cr91} , an open
round hemisphere \cite{Mi}, hyperbolic space \cite{BCG} \cite{Cr04}, and some
spaces of negative curvature \cite{Ot, Cr90} among others. For a discussion on
the subject, see \cite{Cr04}.\\

Recently N. Dairbekov, P. Paternain, P. Stefanov and G. Uhlmann proved (in
\cite{DPSU}) magnetic boundary rigidity between two simple manifolds in several
classes of metrics, including simple conformal metrics, simple analytic
metrics, and all $2-$dimensional simple metrics.

To be able to apply this results to local perturbations of existing metrics we
would like to be able to compare a simple domain with any other (not
necessarily simple) domain with the same boundary behavior. For this we prove
the following theorem. Thus, proving magnetic rigidity for the simple domains
considered in \cite{DPSU}.

\begin{main}
Magnetic simplicity can be read from the metric on the boundary and the
scattering data.
\end{main}

To change the metric in a small region without opening the orbits, it is not
necessary to preserve the scattering data. It could be the case, in principle,
that orbits exit the region in a different place, but after some time, came
back to the region and leave it in the proper place to close up again (see
figure \ref{double orbit}). In section \ref{Rigidity for Surfaces} we look at this case in
$2$ dimensions, and we show the following theorem.

\begin{figure}
\includegraphics{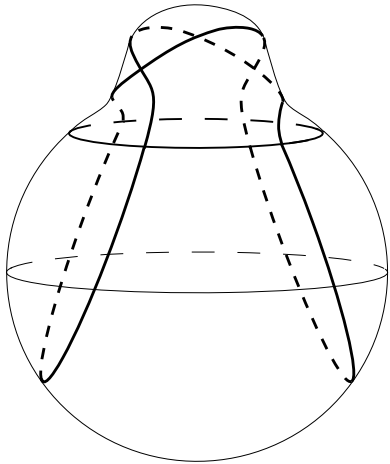}
\caption{A closed orbit going through $R$ twice.} \label{double orbit}
\end{figure}

\begin{main}\label{closed orbits implies scattering}
Let $M$ and $\widehat{M}$ be compact surfaces with magnetic fields, all of
whose magnetic geodesics are closed. Let $R\subset \widehat{M}$ be a strictly
magnetically convex region, such that every magnetic geodesic passes through
$R$ at most once.

If the metric and magnetic fields of $M$ and $\widehat{M}$ agree outside $R$,
then they have the same scattering data.\\
\end{main}

Applying these two theorems for a constant magnetic field on a surface of
constant curvature, we conclude that:

\begin{corM}\label{rigidity for constant curvature}
Given a surface of constant curvature and a fixed $k$ such that all the circles of geodesic curvature $k$ are simple. The surface can not be perturbed in a small enough region while keeping all the curves of geodesic curvature $k$ closed. \\
\end{corM}

\subsection*{Acknowledgments} I would like to thank my advisor Christopher Croke for all his support. \\


\section{Magnetic Jacobi fields and conjugate points}

Let $(M, g)$ be a compact Riemannian manifold with boundary, with a magnetic
field given by the closed $2$-form $\Omega$ on $M$. Denote by $\omega_0$ the
canonical symplectic form on $TM$, that is, the pull back of the canonical
symplectic form of $T^*M$ by the Riemmanian metric.  The geodesic flow can be
described as the Hamiltonian flow of $H$ w.r.t. $\omega_0$, where $H:TM\to \R$
is defined as
$$H(v)=\frac{1}{2}|v|^2_g,\  v\in TM.$$
In a similar way, the magnetic flow $\psi^t:TM\to TM$ can be described as the
Hamiltonian flow of $H$ with respect to the modified symplectic form
$\omega=\omega_0+\pi^*\Omega$. This flow has orbits $t\to
(\gamma(t),\gamma'(t))$, where $\gamma$ is a magnetic geodesic, i.e.
$\nabla_{\gamma'}\gamma'=Y(\gamma')$. Note that when $\Omega =0$ we recover the
geodesic flow, whose orbits are geodesics.\\

It follows from the above definitions that the magnetic geodesics have constant
speed. In fact,
$$\frac{d}{dt}\la\gamma'(t),\gamma'(t) \ra= 2 \la Y(\gamma'),\gamma'\ra= 2\Omega(\gamma',\gamma')=0.$$
Moreover, the trajectories of the magnetic geodesics depend on the energy
level. Unlike geodesics, a rescaling of a magnetic geodesic is not longer a
magnetic geodesic. We will restrict our attention to a single energy level, or
equivalently to unit speed magnetic geodesics. Therefore, from now on, we will
only consider the magnetic flow $\psi^t:SM\to SM$. \\

The choice of energy level is not a restriction, since we can study other
energy levels by considering the form $\tilde{\Omega}=\lambda\Omega$, for any
$\lambda\in \R$.  \\

For $x\in M$ we define the \emph{magnetic exponential map} at $x$ to be the
partial map $exp^\mu_x: T_xM\to M$ given by
$$exp^\mu_x(t\xi)=\pi\circ \psi^t(\xi), \ t\geq 0,\ \xi\in S_xM.$$
This map takes a vector $t\xi\in T_xM$ to the point in $M$ that corresponds to
following the magnetic geodesic with initial direction $\xi$, a time $t$. This
function is $C^\infty$ on $T_xM\setminus \{0\}$ but in general only $C^1$ at
$0$. The lack of smoothness at the origin can be explained by the fact that
magnetic geodesics are not reversible. When we pass through the origin we
change from $\gamma_\xi$ to $\gamma_{-\xi}$, that in general only agree up to
first order. For a proof see Appendix A in \cite{DPSU}.\\

We will say that a point $p\in M$ is \emph{conjugate} to $x$ along a magnetic
geodesic $\gamma$ if $p = \gamma(t_0)= exp^\mu_x(t_0\xi)$ and $v= t_0\xi$ is a
critical point of $exp^\mu_x$. The \emph{multiplicity} of the conjugate point
$p$ is then the dimension of the kernel of $d_v exp^\mu_x$.\\


In what follows, and throughout this paper, if $V$ is a vector field along a
geodesic $\gamma(t)$,  $V'$ will denote the covariant derivative
$\nabla_{\gamma'}V$.

We want to give an alternative characterizations of conjugate points. For this
consider a variation of $\gamma$ through magnetic geodesics. That is
$$f(t,s)=\gamma_{s}(t)$$
where $\gamma_{s}(t)$ is a magnetic geodesic for each $s\in (-\epsilon,
\epsilon)$ and $t\in[0,T]$. Therefore $\frac{D}{\partial t}\frac{\partial
f}{\partial t}= Y(\frac{\partial f}{\partial t})$. Using this and the
definition of the curvature tensor we can write:
$$\begin{array}{rl}
\frac{D}{\partial s}( Y(\frac{\partial f}{\partial t}))
&= \frac{D}{\partial s}\frac{D}{\partial t}\frac{\partial f}{\partial t}
\ =\ \frac{D}{\partial t}\frac{D}{\partial s}\frac{\partial f}{\partial t} -
R(\frac{\partial f}{\partial s},\frac{\partial f}{\partial t})\frac{\partial
f}{\partial t}
\\&=\frac{D}{\partial t}\frac{D}{\partial t}\frac{\partial f}{\partial s}
+ R(\frac{\partial f}{\partial t},\frac{\partial f}{\partial s})\frac{\partial
f}{\partial t}
\end{array}$$

If we call the variational field $J(t)=\frac{\partial f}{\partial s}(t,0)$ we
get for $s=0$
$$\nabla_J(Y(\gamma'))= J'' +R(\gamma',J)\gamma'.$$
 Note also that
$$\nabla_J(Y(\gamma')) = Y(\nabla_J\gamma') + (\nabla_JY)(\gamma')$$
and
\begin{equation}\label{J'nabla}
\nabla_J\gamma'(t)= \frac{D}{\partial s}\frac{\partial f}{\partial
t}(t,0)=\frac{D}{\partial t}\frac{\partial f}{\partial s}(t,0)= J'(t)
\end{equation}
so we can rewrite the above equation as
 $$J''+R(\gamma',J)\gamma'-Y(J')-(\nabla_JY)(\gamma')=0.$$

Since magnetic geodesics can't be rescaled we have $|\gamma_{s}'|=1$ for all
the magnetic geodesics in the variation.  This equation together with equation
(\ref{J'nabla}) gives, for any such variational field $J$,
$$\la J',\gamma'\ra = \la \nabla_J\gamma',\gamma' \ra =
J\la\gamma',\gamma'\ra=0.$$

This equations characterize the variational field of variations through
magnetic geodesics, in a way analogous to the characterization of Jacobi
fields. We will use this equation as a definition as follows. \\

Given a magnetic geodesic $\gamma : [0, T ]\to M$, let $\mathcal{A}$ and
$\mathcal{C}$ be the operators on smooth vector fields $Z$ along $\gamma$
defined by
$$\mathcal{A}(Z) = Z''+R(\gamma',Z)\gamma'-Y(Z')-(\nabla_ZY)(\gamma'),$$
$$\mathcal{C}(Z)=R(\gamma',Z)\gamma'-Y(Z')-(\nabla_ZY)(\gamma').$$

A vector field $J$ along $\gamma$ is said to be a \emph{magnetic Jacobi field}
if it satisfies the equations
\begin{equation}\label{A=0}\mathcal{A}(J) = 0
\end{equation}
and
\begin{equation}\label{J'gamma=0}
\la J',\gamma'\ra =  0.\\
\end{equation}

Note that from equation \ref{A=0} we can see that
$$\begin{array}{rl}
\frac{d}{dt}\la J',\gamma'\ra&= \la J'',\gamma' \ra + \la J',Y(\gamma') \ra
\\&= \la -R(\gamma',J)\gamma'+Y(J')+(\nabla_JY)(\gamma'), \gamma' \ra - \la Y(J'), \gamma' \ra =0\end{array}$$
where we used that $\la(\nabla_JY)(\gamma'), \gamma' \ra =0$ because $Y$ is
skew-symmetric. Therefore, it is enough to check condition \ref{J'gamma=0} at a
point.

Note: In \cite{DPSU} magnetic Jacobi fields were defined without condition \ref{J'gamma=0}, although where used only for vector fields that satisfied this condition. We will add it to the definition since it helps preserve the relation between (magnetic) Jacobi fields and variations through (magnetic) geodesics. \\

A magnetic Jacobi field along a magnetic geodesic $\gamma$ is uniquely
determined by its initial conditions $J(0)$ and $J'(0)$. To see this, consider
the orthonormal basis defined by extending an orthonormal basis $e_1,\dots,e_n$
at $\gamma(0)$ by requiring that
\begin{equation}\label{ON extension}
 e_i'=Y(e_i)
\end{equation}
along $\gamma$. This extension gives an orthonormal basis at each point since
$$ \frac{d}{dt}\la e_i,e_j \ra=\la Y(e_i),e_j \ra+ \la e_i,Y(e_j) \ra =0.$$
Using this basis,
$$z=\sum_{i=1}^n f_ie_i$$
and we can write equation \ref{A=0} as the system
$$f_j'' + \sum_{i=1}^n f_i'y_{ij} + \sum_{i=1}^n f_i a_{ij}=0$$
where $y_{ij}=\la Y(e_i),e_j\ra$ and
$$a_{ij}= \la \nabla_{\gamma'}Y(e_i)+ R(\gamma',e_i)\gamma'-Y(Y(e_i))-(\nabla_{e_i}Y)(\gamma') ,e_j\ra.$$
This is a linear second order system, and therefore it has a unique solution
for each set of initial conditions.\\

Magnetic Jacobi fields correspond exactly to variational field of variations
through magnetic geodesics. In the case of magnetic Jacobi fields $J$ along
$\gamma_{\xi}$ that vanish at $0$ this can be seen by considering
 \begin{equation}\label{variation}
f(t,s)= \gamma_s(t)= exp^\mu_x(t\xi(s))
\end{equation}
where $\xi:(-\epsilon,\epsilon)\to S_xM$ is a curve with $\xi(0)=\xi=
\gamma_{\xi}$ and $\xi'(0)=J'(0)$. This is clearly a variation through magnetic
geodesics, and therefore its variational field $\frac{\partial f}{\partial
s}(t,0)$ satisfies \ref{A=0}. The variational field $\frac{\partial f}{\partial
s}(t,0)$ and the magnetic Jacobi field $J(t)$ are then solutions of \ref{A=0}
with the same initial conditions, therefore they must agree.

For magnetic Jacobi fields $J$ that do not vanish at $0$, we can use the
variation
$$f(t,s)= \gamma_s(t)= exp^\mu_{\tau(s)}(t\xi(s))$$
where $\tau(s)$ is any curve with $\tau'(0)=J(0)$ and $\xi(s)$ is a vector
field along $\tau$ with $\xi(0)=\gamma'(0)$ and $\xi'(0)=J'(0)$.\\

It is easy to see from  the definition and the equation $\gamma''=Y(\gamma')$
that $\gamma'$ is always a magnetic Jacobi field. Unlike the case of straight
geodesics, this is the only magnetic Jacobi field parallel to $\gamma'$.
Another difference from the straight geodesic case is that magnetic Jacobi
fields that are perpendicular to $\gamma'$ at $t=0$ don't stay perpendicular
for all $t$. For this reason we will sometimes consider instead the orthogonal
projection $J^\perp =J-f\gamma'$ where $f=\la J,\gamma'\ra$. The component
$f\gamma'$ of $J$ parallel to $\gamma'$ is uniquely determined by $J^\perp$ and
$J(0)$, since
$$f'= \la J',\gamma'\ra+\la J,\gamma''\ra= \la J, Y(\gamma') \ra= \la J^\perp, Y(\gamma')
\ra.$$\\

We will need one more property of magnetic Jacobi fields that vanish at $0$.
Let $\gamma:[0,T]\to M$ be a magnetic geodesic with $\gamma(0)=x$. Let $v\in
T_{\gamma'}S_xM$, or equivalently under the usual identification, $v\in T_xM$
perpendicular to $\gamma'$.

Let $f(t,s)$ be a  variation through magnetic geodesics of the form
\ref{variation} where $t\in[0,T]$ and $\xi:(-\epsilon, \epsilon)\to S_xM$ with
$\xi(0)=\gamma'(0)$ and $\xi'(0)=v$. The variational field $J_v$ of this
variation is
\begin{equation}\begin{array}{rl}\label{Jv}
J_v(t)&=\frac{\partial f}{\partial s}(t,0)=\left.\frac{\partial }{\partial s}[\pi\circ
\psi^t(\xi(s))]\right|_{s=0}=d_{\xi(0)}[\pi\circ \psi^t](\xi'(0)) \vspace*{2mm}
\\&=d_{\psi^t(\xi)}\pi\circ d_{\xi}\psi^t(v)=d_{\gamma'(t)}\pi\circ d_{\gamma'(0)}\psi^t(v)\end{array}\end{equation}
and its derivative is given by
$$\begin{array}{rl}J_v'(t)&=\frac{D}{\partial t}\frac{\partial f}{\partial s}(t,0)
=\left.\frac{D}{\partial s}\frac{\partial }{\partial t}[\pi\circ
\psi^t(\xi(s))]\right|_{s=0} \vspace*{2mm}\\&= \left.\frac{D}{\partial
s}[\psi^t(\xi(s))]\right|_{s=0}=d_{\xi(0)}\psi^t(\xi'(0))=d_{\gamma'(0)}\psi^t(v).\end{array}
$$
This equations are independent of the variation $f$.\\

Since the magnetic flow is a Hamiltonian flow with respect to the symplectic
form $\omega=\omega_0+\pi^*\Omega$, this form is invariant under the magnetic
flow $\psi^t$ \cite[pg. 10]{Pa}. Therefore for any two magnetic Jacobi fields
$J_v$ and $J_w$ as above, we have that
$$\begin{array}{rl}\omega(d_{\gamma'}\psi^t(v),d_{\gamma'}\psi^t(w)
)&=\omega_0(d_{\gamma'}\psi^t(v),d_{\gamma'}\psi^t(w))+\pi^*\Omega(d_{\gamma'}\psi^t(v),d_{\gamma'}\psi^t(w))\vspace*{2mm} \\&= \la J_v(t),J_w'(t)\ra -\la J_v'(t),J_w(t)\ra +\Omega(J_v(t),J_w(t))\end{array}$$
is independent of $t$. Using also that $J_v(0)=0$, we get
\begin{equation}\label{wJJ=0}
\la J_v,J_w' \ra - \la J_v',J_w \ra + \la Y(J_v),J_w \ra=0
\end{equation}
for any two such Jacobi fields. \\

We will now relate the concepts of magnetic Jacobi fields and conjugate points.

\begin{prop}
Let $\gamma_\xi:[0,T]\to M$ be the magnetic geodesic with $\gamma(0)=x$ and
$\gamma'(0)=\xi$. The point $p=\gamma(t_0)$ is conjugate to $x$ along $\gamma$
if and only if there exist a magnetic Jacobi field $J$ along $\gamma$, not
identically zero, with $J(0)=0$ and $J(t_0)$ parallel to $\gamma'$.

Moreover, the multiplicity of $p$ as a conjugate point is equal to the number
of linearly independent such Jacobi fields.
\end{prop}

Consider a variation through magnetic geodesics as in (\ref{variation}), with
$\xi'(0)=v$ perpendicular to $\gamma'$. Then
$$J_v(t)=\frac{\partial f}{\partial s}(t,0)=d_{t\xi}exp^\mu_x(tv)$$
is a nontrivial magnetic Jacobi field. If there is a vector $v$ for which
$J_v(t_0)$ is parallel to $\gamma'$, then $d_{t_0\xi} exp^\mu_x(t_0v)$ and
$d_{t_0\xi}exp^\mu_x(\xi)=\gamma'$ will be parallel, and $t_0\xi$ is a critical
point of $exp^\mu_x$. Conversely, if $t_0\xi$ is a critical point there must be
a vector $v$  such that $d_{t_0\xi}exp^\mu_x(v)=0$. Let $v^\perp = v - \la
v,\xi\ra\xi$, this is not $0$ since $d_{t_0\xi}exp^\mu_x(\xi)=\gamma'\neq 0$,
and
$$J_{v^\perp}(t_0)= d_{t_0\xi} exp^\mu_x(t_0v) - d_{t_0\xi} exp^\mu_x(t_0\la v,\xi\ra\xi) = - t_0\la
v,\xi\ra\gamma'(t_0).$$

To prove the second statement, note that Jacobi fields $J_{v_i}$ as above are
linearly independent iff the vectors $v_i$ are. Since all $v_i$ are
perpendicular to $\gamma'(0)$, the number of linearly independent vectors will
be the dimension of the kernel of $d_{t_0\xi}exp^\mu_x$, that is the
multiplicity of the conjugate point. \\


\section{The Index form}

Let $\Lambda$ denote the $\mathbb{R}$-vector space of piecewise smooth vector
fields $Z$ along $\gamma$. Define the quadratic form $Ind : \Lambda \to
\mathbb{R}$ by
$$Ind_\gamma(Z) = \int_0^T\{|Z'|^2-\la\mathcal{C}(Z),Z\ra-\la Y(\gamma'),Z\ra^2 \}dt.$$
Note that
$$Ind_\gamma(Z) = -\int_0^T\{\la\mathcal{A}(Z),Z\ra+\la
Y(\gamma'),Z\ra^2 \}dt+\left.\la Z,Z' \ra\right|_0^T +\sum \la Z,Z'^--Z'^+ \left.\ra\right|_{t_i}.$$
where $Z'^\pm$ stands for the left and right derivatives of
$Z$ at the points $t_i$ where the derivative is discontinuous. \\

The $Ind_\gamma(Z)$ generalizes the index form of a geodesic in a Riemannian
manifold. It is easy to see that when $\Omega=0$ these are the same form. We
will see throughout this section that, when restricted to orthogonal vector
fields, we retain some of the relations between (magnetic) Jacobi fields, index
form and conjugate points.\\

Let $\Lambda_0$ denote the $\mathbb{R}$-vector space of piecewise smooth vector
fields $Z$ along $\gamma$ such that $Z(0) = Z(T) = 0$, $\Lambda^\perp$ the
subspace of piecewise smooth vector fields that stay orthogonal to $\gamma'$,
and $\Lambda_0^\perp=\Lambda_0\cap\Lambda^\perp$.\\

For any magnetic Jacobi field $J$ along a magnetic geodesic $\gamma$, let
$f=\la J, \gamma'\ra$ and $J^\perp=J-f \gamma'$ the component of $J$ orthogonal to
$\gamma$ , using that $\gamma''=Y(\gamma')$  we have\\
$\mathcal{A}(J)= $
$${J^\perp}''+f''\gamma' +2f'Y(\gamma')+f\gamma'''+R(\gamma',J^\perp+f \gamma')\gamma'
-Y({J^\perp}'+f'\gamma'+f\gamma'')-(\nabla_{Z+f\gamma'}Y)(\gamma')$$
so\begin{equation}\label{AofJasZf} 0=A(J^\perp)+fA(\gamma')+f''\gamma'+f'Y(\gamma')
\end{equation}
On the other hand, any magnetic Jacobi field satisfies $\la J', \gamma'\ra=0$.
Using this together with $f=\la J, \gamma'\ra$ and $\gamma''=Y(\gamma')$ we see
that $f'= \la J, Y(\gamma')\ra=\la J^\perp, Y(\gamma')\ra$. Since $\gamma'$ is a
Jacobi field, $A(\gamma')=0$, and we have from (\ref{AofJasZf})
$$\la A(J^\perp), Z\ra =-f'\la Y(\gamma'),J^\perp\ra= -\la Y(\gamma'),J^\perp\ra^2$$
therefore if $J$ is a magnetic Jacobi field and its orthogonal component $J^\perp$ is
in $\Lambda_0$
\begin{equation}\label{Ind(J)=0}
Ind_\gamma(Z) = -\int_0^T\{\la\mathcal{A}(J^\perp),J^\perp\ra+\la Y(\gamma'),J^\perp\ra^2 \}dt
=0.
\end{equation}

Also, if there are no conjugate points along $\gamma$, we can restrict our
attention to a neighborhood of $\gamma$ for which the magnetic exponential is a
diffeomorphism. In this case we can adapt the proof given in \cite{DPSU} for
the case of simple domains to prove the following version of the Index Lemma.
We include the proof below.

\begin{lemma}\label{Index Lemma}
Let $\gamma$ be a magnetic geodesic without conjugate points to $\gamma(0)$.
Let $J$ be a magnetic Jacobi field, and $J^\perp$ the component orthogonal to
$\gamma'$. Let $Z\in \Lambda^\perp$ be a piecewise differentiable vector field
along $\gamma$, perpendicular to $\gamma'$. Suppose that
\begin{equation}\label{endpoints}
J^\perp(0)=Z(0)=0 \text{ and } J^\perp(T)=Z(T) \end{equation} Then
$$Ind_\gamma(J^\perp)\leq Ind_\gamma(Z)$$
and equality occurs if and only if $Z=J^\perp$.\\
\end{lemma}

Note that in the case where the vector field $Z$ satisfies $Z(0)=Z(T)=0$, $J=
\gamma'$ is a Jacobi field that satisfies the above hypothesis, giving the
following corollary.

\begin{cor}\label{cor to Index Lemma}
If $\gamma$ has no conjugate points and $Z\in \Lambda^\perp_0$, then
$Ind_\gamma(Z)\geq 0$, with equality if and only if $Z=0$.
\end{cor}

In other words, if $\gamma$ has no conjugate points, the quadratic form
$Ind_\gamma : \Lambda^\perp_0\to \mathbb{R}$ is positive definite. \\

\begin{proof}[Proof of Lemma \ref{Index Lemma}]

Given a vector $v \in T_{\gamma(0)}M$ orthogonal to $\gamma'$, we can define a
magnetic Jacobi field $J_v$ along $\gamma$ as in (\ref{Jv}) that has $J_v(0)=0$
and $J_v'(0)=v$. Since there are no conjugate points to $\gamma(0)$ along
$\gamma$, these Jacobi fields are never $0$ nor parallel to $\gamma'$. There
exist a basis $\{v_1, \dots, v_n\}$ of $T_{\gamma(0)}M$, with $v_1=\gamma'$
such that $J_1(t)=\gamma'(t)$ and $J_i(t)=J_{v_i}(t)$, $i=2,\dots,n$ form a
basis for $T_{\gamma(t)}M$ for  all
$t\in (0,T]$. \\

If $Z$ is a vector field with $Z(0)=0$, we can write, for $t\in (0,T]$
$$Z(t)=\sum_{i=1}^n f_i(t) J_i(t)$$
where $f_1,\dots,f_n$ are smooth functions. This function can be smoothly
extended to $t=0$, as we now show. For $i\geq 2$, we can write $J_i(t)=tA_i(t)$
where $A_i$ are smooth vector fields with $A_i(0)=J_i'(0)$. Then each $A_i(t)$
is parallel to $J_i(t)$ and $\{\gamma'(t), A_2(t), \dots, A_n(t)\}$ is a basis
for all $t\in [0,T]$, so
$$Z(t)= g_1\gamma' +\sum_{i=2}^n g_i(t) A_i(t).$$
It follows that for $t\in (0,T]$, $g_1=f_1$ and $g_i(t)=t f_i(t)$ for $i\geq
2$, and since $Z(0)=0$, $g_i(0)=0$ and $f_i$ extends smoothly to $t=0$.\\

Using this representation we can write\\

$Ind_\gamma(Z) =$ $$-\sum_{i,j}\int_0^T\la\mathcal{A}(f_iJ_i),f_jJ_j\ra dt-\int_0^T\la
Y(\gamma'),Z\ra^2 dt+\left.\la Z,Z' \ra\right|_0^T +\sum \la Z,Z'^--Z'^+ \left.\ra\right|_{t_k}.$$
where $Z'^\pm$ stands for the left and right derivatives of
$Z$ at the points $t_k$ where the derivative is discontinuous. \\

And
\begin{equation}\label{coordinates}
\begin{array}{rl}\mathcal{A}(f_iJ_i) &=f_i''J_i+2f_i'J_i'+f_iJ_i''+f_iR(\gamma',J_i)\gamma'-f_i'Y(J_i)
-f_iY(J_i')-f_i(\nabla_{J_i}Y)(\gamma')\\ &= f_i\mathcal{A}(J_i)+ f_i''J_i+2f_i'J_i'-f_i'Y(J_i)\end{array}
\end{equation}
where the first term is $0$ since $J_i$ is a magnetic Jacobi field.\\

We know, moreover, from (\ref{wJJ=0}) that
$$\la J_i,J_j' \ra - \la J_i',J_j \ra + \la Y(J_i),J_j \ra=0$$
so we can write
$$\begin{array}{rl}\la \mathcal{A}(f_iJ_i),J_j \ra &= \la f_i''J_i,J_j \ra
+ 2\la f_i'J_i',J_j \ra - \la f_i'Y(J_i),J_j \ra \\ &=f_i'' \la J_i,J_j \ra +  f_i'\la J_i',J_j \ra + f_i'\la J_i,J_j' \ra \\ &=\frac{d}{dt}( f_i' \la J_i,J_j\ra)\end{array}$$
and
$$\int_0^T\la\mathcal{A}(f_iJ_i),f_jJ_j\ra dt = \int_0^T f_j\frac{d}{dt}( f_i' \la J_i,J_j\ra) dt=\left.\la f_i' J_i,f_jJ_j\ra\right|_0^T - \int_0^T\la f_i'J_i,f_j'J_j\ra dt.$$

Using this, and that $Z(0)=0$ we can write the index form as
$$Ind_\gamma(Z) = \int_0^T||\sum_1^n f_i'J_i||^2dt
- \left.\la \sum_1^n f_i' J_i, Z\ra\right|_0^T- \left.\sum \la \sum_1^n ({f_i'}^--{f_i'}^+)J_i,
Z\ra\right|_{t_k}$$
$$\hspace*{4cm}-\int_0^T\la Y(\gamma'),Z\ra^2 dt+\left.\la Z,Z' \ra\right|_0^T +\left.\sum \la Z,Z'^--Z'^+
\ra\right|_{t_k}$$
$$=\!\int_0^T\!||\sum_1^n f_i'J_i||^2-\la Y(\gamma'),Z\ra^2 dt
+\la\! Z(T),\!\sum_1^n \!f_i(T)J_i'(T)\!\ra +\left.\!\sum\! \la\! Z,\!\sum_1^n\!
f_i({J_i'}^-\!-{J_i'}^+\! )\!\ra\right|_{t_k}$$

where the last term is $0$ because $J_i$ is differentiable. \\

Let $W=\sum_2^n f_i'J_i$, and remember that $J_1=\gamma'$, then we get
$$||\sum_1^n f_i'J_i||^2= \la f_1'\gamma'+W,f_1'\gamma'+W \ra =
f_1'^2 + 2f_1' \la \gamma',W \ra + \la W,W \ra.$$

Since $Z$ is orthogonal to $\gamma'$, $\la Z,\gamma' \ra=0$ and by
differentiating $\la Z',\gamma' \ra= -\la Z,Y(\gamma') \ra$. Using also that
for magnetic Jacobi fields $\la J',\gamma'\ra=0$ we have
$$- \la Z,Y(\gamma') \ra =
\la \sum_1^n f_i' J_i+ f_i J_i', \gamma'\ra = f_1' +\la W, \gamma' \ra
$$ and
$$ \la Z,Y(\gamma') \ra^2= f_1'^2 + 2f_1' \la W, \gamma'\ra + \la W, \gamma' \ra^2.$$

so
$$Ind_\gamma(Z) = \int_0^T||\sum_1^n f_i'J_i||^2 - \la
Y(\gamma'),Z\ra^2 dt +\la Z(T),\sum_1^n f_i(T)J_i'(T) \ra$$
$$\hspace*{.2cm} = \int_0^T \la W,W \ra -\la W, \gamma' \ra^2 dt +\la Z(T),\sum_1^n f_i(T)J_i'(T) \ra $$
\begin{equation}\label{Index Eq}
= \int_0^T ||W^\perp|| dt  +\la Z(T),\sum_1^n f_i(T)J_i'(T)\ra \hspace*{1.6cm} \end{equation}
where $W^\perp$ is the component of $W$ orthogonal to $\gamma'$.\\

Let $J$ be the magnetic Jacobi field $J=\sum_2^n f_i(T)J_i$. The fact that it
is a magnetic Jacobi field is easy to see since $\mathcal{A}$ is linear over
$\R$, it is also clear that $J(0)=0$ and $J^\perp(T)=Z(T)$. Since
$$J^\perp =\sum_2^n f_i(T)J_i - \sum_2^n f_i(T)\la J_i,\gamma'\ra \gamma'$$
the corresponding functions $f_i(T)$ are constant for $i\geq2$, so $W=0$ and we
can see from equation (\ref{Index Eq}) that
$$Ind_\gamma(J^\perp)= \la J(T),\sum_1^n f_i(T)J_i'(T) \ra$$
that gives
$$Ind_\gamma(Z)-Ind_\gamma(J^\perp) =\int_0^T ||W^\perp|| dt \geq 0$$
with equality iff $W^\perp$ vanishes everywhere. That is
$$W^\perp=\sum_2^n f_i'J_i- \la W,\gamma'\ra\gamma'=0$$
and therefore $f_i$ constant for $i\geq 2$. So $Z = f_1\gamma' + \sum_2^n
f_iJ_i = f_1\gamma' +J$, and since $Z$ is orthogonal to $\gamma'$ this implies
that $Z=J^\perp$.\\
\end{proof}

In what follows we will want to use the above lemma for more general vector
fields, that do not vanish at $0$ but vanish at $T$. Since magnetic flows are
not reversible, we can't simply reverse time. We will consider instead the
associated magnetic flow $(M,g, -\Omega)$. This magnetic flow has the same
magnetic geodesics, but with opposite orientation.

\begin{lemma}
Let $(M,g, \Omega)$ be a magnetic field. Then $(M,g,-\Omega)$ is also a
magnetic field.
\begin{enumerate}
\item Magnetic geodesics in both magnetic fields agree, but with opposite orientation.
\item Jacobi fields agree in both magnetic fields.
\item Index form is independent of its orientation.
\end{enumerate}
\end{lemma}

If $\gamma:[0,T]\to M$ is a magnetic geodesic in $(M,g, \Omega)$, denote by
$\gamma_-$ the geodesic with opposite orientation, that is
$\gamma_-(t)=\gamma(-t)$, for $t\in [-T,0]$. Then $\gamma_-''(t)= \gamma''(-t)=
Y(\gamma'(-t))=-Y(\gamma_-'(t))$, so $\gamma_-$ is a magnetic geodesic in
$(M,g,-\Omega)$. Part $2$ follows from $1$ and the fact that magnetic Jacobi
fields are variational fields of variations through magnetic geodesics.
Alternatively, we can check that for $J_-(t)=J(-t)$:
$$\begin{array}{rl}\mathcal{A}_-(J_-(t)) &= J_-''(t)+R(\gamma_-'(t),J_-(t))\gamma_-'(t)+Y(J_-'(t))+(\nabla_{J_-}Y)(\gamma_-'(t))\\
&= J''(-t)+R(\gamma'(-t),J(-t))\gamma'(-t)-Y(J'(-t))-(\nabla_JY)(\gamma'(-t))\\ &= \mathcal{A}(J(-t)).\end{array}$$
and $$\la J_-'(t),\gamma_-'(t)\ra= \la-J'(-t), -\gamma'(-t)\ra= \la J'(-t),
\gamma'(-t)\ra. $$\\

From the above computation follows also that
$\mathcal{C}_-(Z_-(t))=\mathcal{C}(Z(t))$. So
$$Ind_{\gamma_-}(Z_-)=\int_{-T}^0\{|Z_-'(t)|^2-\la\mathcal{C_-}(Z_-(t)),Z_-(t)\ra
-\la Y_-(\gamma_-'(t)),Z_-(t)\ra^2\}dt$$
$$\hspace*{1.9cm} = \int_{-T}^0\{|Z'(-t)|^2-\la\mathcal{C}(Z(-t)),Z(-t)\ra-\la Y(\gamma'(-t)),Z(-t)\ra^2 \}dt$$
$$\hspace*{2.3cm} = \int_{0}^T\{|Z'(t)|^2-\la\mathcal{C}(Z(t)),Z(t)\ra-\la Y(\gamma'(t)),Z(t)\ra^2 \}dt= Ind_\gamma(Z).$$\\

\begin{cor}\label{reversed Index Lemma}
Lemma \ref{Index Lemma} holds when we replace equation \eqref{endpoints} with
$Z(0)=J^\perp(0)$ and $Z(T)=J^\perp(T)=0$.
\end{cor}

When $Z(T)=J^\perp(T)=0$, we can consider $Z_-$ and $J^\perp_-$, this will
satisfy the hypothesis of lemma \ref{Index Lemma}, so we have:
$$Ind_\gamma(J^\perp)=Ind_{\gamma_-}(J_-^\perp)\leq Ind_{\gamma_-}(Z_-)=Ind_\gamma(Z).$$\\

\begin{lemma}\label{Negative Index}
If $\gamma(t_0)$ is conjugate to $\gamma(0)$ along $\gamma$, for some $t_0<T$,
then there is a vector field $Z\in \Lambda^\perp_0$  with $Ind_\gamma(Z)<0$.

\end{lemma}

Let $J$ a Jacobi field along $\gamma$ with $J(0)=J(t_0)=0$, and $\tilde{J}$ be
$J^\perp$ for $t\in [0,t_0]$ and $0$ for $t\in [t_0,T]$. Then
$Ind_\gamma(\tilde{J})=0$. We can use the Index Lemma and corollary
\ref{reversed Index Lemma} to show that cutting the corner at $t_0$ by
replacing $\left.\tilde{J}\right|_{[t_0-\epsilon, t_0+\epsilon]}$ by a Jacobi field with
the same endpoints decreases the value of the Index form. So
this new vector field has  $Ind_\gamma(Z)< 0$.\\

\begin{lemma}\label{Index at conjugate point}
If $\gamma(T)$ is the first conjugate point to $\gamma(0)$ along $\gamma$ and
$Z\in \Lambda^\perp_0$, then $Ind_\gamma(Z)\geq 0$, with equality if and only
if $Z=0$ or $Z$ is the perpendicular component of a Jacobi field.
\end{lemma}

This is an extension of lemma \ref{Index Lemma}, and it is proved by a similar
argument.We will follow the proof of lemma \ref{Index Lemma}, using the same
notation. Suppose $\gamma(T)$ is a first conjugate point and has multiplicity
$k$, then we can find a basis $\{v_0, \dots, v_{n-1}\}$ of $T_{\gamma(0)}M$,
with $v_0=\gamma'$ such that $J_{v_i}(T)$ are parallel to $\gamma'(T)$ for
$i=1,\dots,k$ and are not parallel to  $\gamma'(T)$ for $i=k+1,\dots,n-1$. Then
$J_0(t)=\gamma'(t)$ and $J_i(t)=J_{v_i}(t)$, $i=1,\dots,n-1$ form a basis for
$T_{\gamma(t)}M$ for all $t\in (0,T)$.

If $Z$ is a vector field in $\Lambda_0^\perp$, we can write, for $t\in (0,T)$
$$Z(t)=\sum_{i=0}^{n-1} f_i(t) J_i(t)$$
where $f_0,\dots,f_{n-1}$ are smooth functions. We can extend these functions
to $t=0$ as before. To extend $f_i$ to $t=T$ we can write
$J_i=(t-T)A_i+J_i(T)$, for $i=1, \dots, k$. Then $A_i$ are smooth vector fields
with $A_i(T)=J_i'(T)$ that is orthogonal to $\gamma'$ since $\la J',\gamma'
\ra=0$ for all Jacobi fields. Then $\{\gamma', A_1, \dots, A_k, J_{k+1},\dots,
J_{n-1}\}$ are a basis for all $t\in (0,T]$, and
$$Z(t)= g_0\gamma' +\sum_{i=1}^k g_i(t) A_i(t)+\sum_{i=k+1}^{n-1} g_i(t) J_i(t).$$
It follows that for $t\in (0,T)$, $g_i=f_i$ for $i>k$, $g_i(t)=(t-T) f_i(t)$
for $0<i\leq k$, and $g_0=f_0+\sum_{i=1}^k f_i(t) \la J_i(T),\gamma'\ra $.
Since $Z(T)=0$, $g_i(T)=0$ and $f_i$ extends smoothly to $t=T$.\\

Following the proof of lemma \ref{Index Lemma}, we get from (\ref{Index Eq})

$$Ind_\gamma(Z)= \int_0^T ||W^\perp|| dt  +\la Z(T),\sum_1^n f_i(T)J_i'(T) \ra$$
$$= \int_0^T ||W^\perp|| dt \geq 0 \hspace*{2.1cm}$$
with equality iff $W^\perp$ vanishes everywhere. That is when $f_i$ constant
for $i>0$. So $Z = f_1\gamma' +J$ for some Jacobi field $J$, and since $Z$ is
orthogonal to $\gamma'$ this implies that $Z=J^\perp$.\\

\begin{cor}
$Ind_\gamma(Z)$ restricted to $\Lambda_0^\perp$ is positive definite if and
only if $\gamma$ has no conjugate points.
\end{cor}

When $\gamma$ has no conjugate points, it follows directly from corollary
\ref{cor to Index Lemma} that $Ind_\gamma$ is positive definite. In the case
that the endpoints are conjugate to each other $Ind_\gamma$ has nontrivial
kernel, as can be seen from equation \ref{Ind(J)=0}. If $\gamma$ has conjugate
points, we saw on lemma \ref{Negative Index} that there is a vector field in
$\Lambda_0^\perp$ with $Ind_\gamma <0$, therefore it is not positive
definite.\\

We will be interested in the dependence of the index form on its parameters.
For this consider a continuous (possibly constant) family of vectors $\xi(s)\in
S_xM$ and the correspondent family of magnetic geodesics $\gamma_s(t)=
exp^\mu_x(t\xi(s))$. Let $T_s$, the length of each geodesic, be continuous on
$s$. Let $\Lambda_s$ denote the vector space of piecewise smooth vector fields
$Z_s$ along $\gamma_s$, perpendicular to $\gamma_s'$ and such that $Z(0) =
Z(T_s) = 0$.

Let $\{v_1,\dots,v_n\}$ be an orthonormal basis with $v_1=\gamma_0'(0)$, and
extend it to a continuous family  $\{v_1(s),\dots,v_n(s)\}$ of orthonormal
basis for each $s$ with  $v_1(s)=\xi(s)$. This can be done by defining
$$v_i(s)=\rho_s(v_i)$$
where $\rho_s$ is a rotation of $S^{n}$ with $\rho_s(\xi(0))= \xi(s)$. We
extend this for all $t$ by requiring that
\begin{equation}
 \nabla_{\gamma_s'}e_i=Y(e_i)
\end{equation}
along each magnetic geodesic. As in \ref{ON extension} this gives an
orthonormal basis for each point.

Using this basis, we can extend any vector field $Z=\sum_{2}^n a_i(t)e_i(t)$ in
$\Lambda_0$ to a vector field over the family of geodesics by
$$Z(s,t)=\sum_{2}^n a_i(t\frac{T_0}{T_s})e_i(s,t),$$
that belongs to $\Lambda_s$ when restricted to each $\gamma_s$.  We will denote
the set of such vector fields by $\Lambda_{[0,1]}$.

Since everything depends continuously on $s$, so does
$$Ind_{\gamma_s}(Z) = \int_0^{T_s}\{|Z'|^2-\la\mathcal{C}(Z),Z\ra-\la Y(\gamma_s'),Z\ra^2
\}dt.$$\\

We will be mostly interested on whether the Index form is positive definite.
For such a family of curves, the fact that the Index form  is positive definite
(and therefore the non existence of conjugate points) depends continuously on
$s$ in the following sense. If the index form is positive definite for some
$s_0$, and has a negative value for some $s_1$ there must be some $s\in
(s_0,s_1)$ where it has non-trivial kernel. Moreover, the first such $s$ will
be when
$\gamma_s$ has conjugate endpoints and no conjugate points in the interior.\\

\section{Simple metrics and boundary data}\label{Simple metrics and boundary data}

Consider a manifold $M_1$ such that $M\subset int(M_1)$, extend $g$ and
$\Omega$ smoothly. We say that $M$ is \emph{magnetic convex} at $x\in \partial
M$ if there is a neighborhood $U$ of $x$ in $M_1$ such that all unit speed
magnetic geodesics in $U$, passing through $x$ and tangent to $\partial M$ at
$x$, lie in $U \setminus int(M)$. It is not hard to see that this definition
depend neither on the choice of $M_1$ nor on the way we extend $g$ and $\Omega$
to $M_1$.\\

Let $\text{II}$ stand for the second fundamental form of $\partial M$ and
$\nu(x)$ for the inward pointing normal. Then if $M$ is magnetic convex
$$\text{II}(x,\xi)\geq \la Y_x(\xi),\nu(x)\ra $$ for all $(x,\xi)\in TM$. \cite[Lemma A.6]{DPSU}.

We say that $\partial M$ is \emph{strictly magnetic convex} if
$$\text{II}(x,\xi)> \la Y_x(\xi),\nu(x)\ra
$$ for all $(x,\xi)\in TM$.

This condition implies that the tangent geodesics do not intersect $M$ except
for $x$, as shown in \cite[Lemma A.6]{DPSU}.\\

We say that $M$ is \emph{simple}(w.r.t.$(g,\Omega)$) if $\partial M$ is
strictly magnetic convex and the magnetic exponential map $exp^\mu_x:
(exp^\mu_x)^{-1}(M)\to M$ is a diffeomorphism for every $x\in M$.\\

For $(x,\xi)\in SM$, let $\gamma_{\xi} : [l^-(x,\xi), l(x,\xi)]\to M$ be the
magnetic geodesic such that $\gamma_{\xi}(0) = x$, $\gamma_{\xi}'(0) = \xi$,
and $\gamma_{\xi}(l^-(x,\xi)), \gamma_{\xi}(l(x,\xi))\in\partial M$. Where $l-$
and $l$ can take the values $\pm \infty$ if the magnetic geodesic
$\gamma_{\xi}$ stays in the interior of $M$ for all time in the corresponding
direction.\\

Let $\partial_+SM$ and $\partial_-SM$ denote the bundles of inward and outward
unit vectors over $\partial M$:
$$\partial_+SM=\{(x,\xi)\in SM : x\in\partial M, \la\xi,\nu(x)\ra \geq 0\},$$
$$\partial_-SM=\{(x,\xi)\in SM : x\in\partial M, \la\xi,\nu(x)\ra \leq 0\},$$
where $\nu$ is the inward unit normal to $\partial M$. Note that $\partial(SM)
=
\partial_+SM \cup \partial_-SM$ and $\partial_+SM \cap \partial_-SM=S(\partial M)$.\\

In the case that $M$ is simple, is clear that the functions $l^-(x,\xi)$ and
$l(x, \xi)$ are continuous and, on using the implicit function theorem, they
are easily seen to be smooth near a point $(x, \xi)$ such that the magnetic
geodesic $\gamma_{\xi}(t)$ meets $\partial M$ transversely at $t = l^-(x, \xi)$
and $t = l(x, \xi)$ respectively. By the definition of strict magnetic
convexity, $\gamma_{\xi}(t)$ meets $\partial M$ transversely for all $(x,\xi)
\in SM \setminus S(\partial M).$ In fact, these functions are smooth
everywhere, as was shown by Dairbekov, Paternain, Stefanov and Uhlmann in the
following lemma.\\

\begin{lemma}\cite[Lemma 2.3]{DPSU}\label{lsmooth}
For a simple magnetic system, the function $\mathbb{L}:\partial (SM)\to \R$,
defined by
$$\mathbb{L}(x,\xi):=\begin{cases}l(x,\xi) & if (x,\xi)\in\partial_+SM  \\
l^-(x,\xi) & if (x,\xi)\in\partial_-SM  \end{cases}  $$ is smooth. In
Particular, $l:\partial_+SM \to\R$ is smooth. The ratio
$$\frac{\mathbb{L}(x,\xi)}{\la\nu(x),\xi\ra}$$
is uniformly bounded on $\partial(SM)\setminus S(\partial M)$.\\
\end{lemma}

This lemma was proved as stated, for simple magnetic systems, but the proof is
a local argument using only the strong magnetic convexity of the region.\\

The \emph{scattering relation} $\mathcal{S} :\partial_+SM \to \partial_-SM$ of
a magnetic system $(M, g, \Omega)$ is defined as follows:
$$\mathcal{S}(x,\xi) = (\gamma_{\xi}(l(x,\xi)),\gamma_{\xi}'(l(x,\xi)))$$
when the value $l(x,\xi)$ is finite, otherwise it is not defined. \\

\begin{figure}
\psfrag{v}{$\xi$}
     \psfrag{Sv}{$\mathcal{S}(\xi)$}
\includegraphics[scale=0.8]{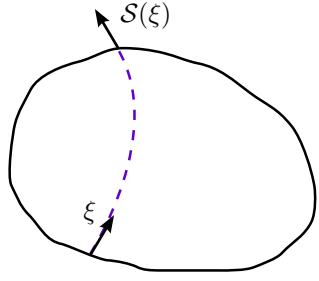}
\caption{The Scattering relation.} \label{Scattering}
\end{figure}

The \emph{restricted scattering relation} $\texttt{s} :\partial_+SM \to \partial M$ is
defined to be the postcomposition of the scattering relation with the natural
projection of $\partial_-SM$ to $\partial M$, i.e., $$\texttt{s}(x, \xi)
=\gamma_{\xi}(l(x,\xi))$$ when properly defined.\\

We are interested only in simple domains, and domains that have the same
scattering data as a simple domain, so we will assume that $l^-$ and $l$ are
finite and smooth on $\partial (SM)$. Moreover, it follows from the smoothness
of $l$ and their definitions that both $\mathcal{S}$ and $\texttt{s}$ are
smooth everywhere on $\partial_+SM$.\\

Let $\widehat{M}$ be a compact simple domain with respect to $\widehat{\Omega}$
in the interior of a manifold $(\widehat{M_1},\hat{g})$. Let $M$ be a compact
domain in the interior of a manifold $(M_1,{g})$ with $\Omega$. Related in such
way that $\hat{g}=g$ and $\widehat{\Omega}=\Omega$ on $\widehat{M_1}\setminus
\widehat{M} =M_1\setminus M$, and the (restricted) scattering relations
$\widehat{\mathcal{S}}, \mathcal{S}$ agree on $\partial M (=\partial
\widehat{M})$. To be able to compare the magnetic flows in $M$ and
$\widehat{M}$ we would like to say that $M$ is also simple, without having to
impose it as a condition. The purpose of this section is to prove the following
theorem.

\begin{thm}\label{simple implies simple}
Given $M$ and $\widehat{M}$ as above, then $M$ is also simple.
\end{thm}

To prove that the magnetic exponential is a diffeomorphism we need to show
that it has no conjugate points. For this we need the following lemma.\\

\begin{lemma}\label{conj on boundary}
If there are conjugate points in $M$, then there is a pair of points in
$\partial M$ conjugate to each other.
\end{lemma}

Suppose there is a point in the interior conjugate to $x\in \partial M$ along a
geodesic $\gamma_\xi$. Let $\tau:[0,1]\to S_xM$ be a curve joining $S_x\partial
M$ to $\xi$, and consider the family of magnetic geodesics
$\gamma_s=\exp_x^\mu(t\tau(s))$. These geodesics exit $M$ at time
$l(x,\gamma_s')$, that by the simplicity of $\widehat{M}$ is a continuous
function of $s$. Close enough to $x$ the magnetic exponential is a
diffeomorphism, and by lemma \ref{lsmooth} there is a $C>0$ such that
${l(x,\eta)}\leq C{\la\nu(x),\eta\ra}$ for all $\eta \in  S_xM$. This implies
that for $s$ small enough, the magnetic geodesic from $x$ to
$\texttt{s}(\tau(s))$ is short, and stays inside a neighborhood where the
magnetic exponential is a diffeomorphism. Therefore, it has no conjugate
points, and the index form is positive definite close to $x$. On the other
hand, there is a perpendicular vector field along $\gamma_\xi$ for which the
index form is negative. Then $Ind_{\gamma_s}$ is positive definite for $s=0$
and not for $s=1$. Let $s_0$ be the smallest $s$ for which $Ind_{\gamma_s}$ has
non trivial kernel. Then, by the results on the previous section,
$\texttt{s}(\tau(s_0))$ is conjugate to $x$ along the magnetic geodesic
$\gamma_s$ that joins them.

If there are points conjugate to each other along a magnetic geodesic
$\gamma_\xi$, and both lie in the interior of $M$, there must be a point
conjugate to $\gamma_\xi(0)$ along this magnetic geodesic. Therefore reducing
the problem to the case above. This can be proved by a similar argument using
the family of geodesics $\left.\gamma_{\xi}\right|_{[0,sT]}$.

\begin{proof}[Proof of Theorem \ref{simple implies simple}]

It is easy to see from the definition that the domain $M$ has to be strictly
magnetic convex, since the metrics and magnetic flows agree outside $M$.\\

To prove that the magnetic exponential map is a diffeomorphism form
$(exp^\mu_x)^{-1}(M)$ to $M$ we need to show that it has no conjugate points,
i.e. there are no points in $M$ that are conjugate to each other along a
magnetic geodesic. For this purpose assume such points exist, then by Lemma
\ref{conj on boundary} there are points $x, y \in
\partial M$ conjugate to each other along a magnetic geodesic $\gamma_\xi$,
where $\gamma_\xi(0)=x$ and $\gamma_\xi(t_0)=y $ for some $t_0 > 0$. \\

Let $J$ be a magnetic Jacobi field along $\gamma_{\xi}$ that vanishes at $0$,
and $f(s,t)$ a variation through magnetic geodesics with $f(0,t)=\gamma_\xi$
and $J$ as a variational field. We can use $f(s,t)=\gamma_s(t)=exp^\mu_x(t
\xi(s))$ where $\xi:(-\epsilon,\epsilon)\to S_xM$ is a curve with $\xi(0)=\xi$,
$\xi'(0)=J'(0)$. $f$ is well defined in $M$ for $(s,t)\in
(-\epsilon,\epsilon)\times [0,T_s]$ where $T_s=l(x,\xi(s))$. Consider
$c(s)=f(s,T_s)\in \partial M$ the curve of the exit points in $\partial M$.
Then
$$\frac{dc}{ds}(0)=\frac{df}{ds}(0,T_s)+\frac{df}{dt}\frac{dT_s}{ds}(0,T_s)$$
$$\hspace*{.5cm}= J(T_s)+\frac{dT_s}{ds}(0,T_s)\gamma_\xi'.$$

If $\gamma_\xi(l(x,\xi))$ is conjugate to $\gamma_\xi(0)$ along $\gamma_\xi$,
there is a Jacobi field $J$ that is $0$ at $t=0$ and parallel to $\gamma_\xi$
at $T_s$, then $\frac{dc}{ds}(0)$ is parallel to $\gamma_\xi'$. On the other
hand, if $\frac{dc}{ds}(0)$ is parallel to $\gamma_\xi'$ for any Jacobi field
with $J(0)=0$ then $J(T_s)$ is parallel to $\gamma_\xi'$. Therefore
$\gamma_\xi(T_s)$ is conjugate to $\gamma_\xi(0)$ along $\gamma_\xi$.

Note that we can write $c(s)=\texttt{s}(\xi(s))$, that depends only on the
scattering data, so the scattering relation detects conjugate points in the
boundary. Since there are no conjugate points in the boundary of $\widehat{M}$,
there can be none in $M$. Therefore the magnetic exponential is a local
diffeomorphism.

We will now see that $exp^\mu_x$ is a global diffeomorphism from $(exp^\mu_x)^{-1}(M)$ to $M$. To see that it is surjective let $x\in\partial M$, and $y$ any point in $M$. Let $c:[0,1]\to M$ be a path from $x$ to $y$, and consider the set $A\subset[0,1]$ of points such that $c(s)$ is in the image of $exp^\mu_x$. This set is open, since $exp^\mu_x$ is a local diffeomorphism. To see that it is closed, choose a sequence $s_n\in A$ converging to $s_0$. Then $c(s_n)= exp^\mu_x(t(s_n)\xi(s_n))$, and there is a subsequence such that $t(s_n)$ and $\xi(s_n)$ converge to $t_0$ and $\xi_0$ respectively. If $t_0\xi_0\notin (exp^\mu_x)^{-1}(M)$, there must be a first $t_1< t_0$ such that $exp^\mu_x(t_1\xi_0)\in \partial M$. Then $exp^\mu_x(t\xi_0)$ must be tangent to $\partial M$ and inside $M$ for $t<t_1$, which contradicts the magnetic convexity of $M$. Then, $A$ is both open and closed, therefore $A=[0,1]$ and $y$ is in the image of $exp^\mu_x$.

To see that $exp^\mu_x$ is injective for $x\in\partial M$, note that it is a covering map. The point $x$ has only one preimage, since by the simplicity of $\widehat{M}$ there are no magnetic geodesics form $x$ to $x$. Therefore $exp^\mu_x$ is a covering map of degree $1$.

To prove this for $x\notin\partial M$, we need to see that there are no trapped magnetic geodesics, that is, that there are no magnetic geodesics that stay inside $M$ for an infinite time. Note that, since any magnetic geodesic that enters the region at $\xi$ has to exit at $\texttt{s}(\xi)$, it is enough to see that all geodesics enter the region at a finite time.  Let $\gamma$ be a magnetic geodesic. We know that we can reach the point $\gamma(0)$ from the boundary, so there is a variation through magnetic geodesics $\gamma_s(t)$ with $\gamma_0=\gamma$, $\gamma_s(0)=\gamma(0)$ for all $s\in [0,1]$, and $\gamma_1(t_1)\in \partial M$ for some $t_1<0$. If $\gamma_{s_0}$ intersects $\partial M$, by the magnetic convexity of $M$ it has to be a transverse intersection, therefore intersecting $\partial M$ is an open condition on $[0,1]$. It is also a closed condition, by continuity of the geodesic flow and compactness of $\partial M$. Therefore, since $\gamma_1$ intersects $\partial M$, so does $\gamma_{s}$ for all $s$, and $\gamma$ is not trapped.

Now we see that $exp^\mu_x$ is a global diffeomorphism from $(exp^\mu_x)^{-1}(M)$ to $M$ for $x\notin\partial M$. Since $exp^\mu_x$ is injective for $x\in\partial M$, and all geodesics come from some point $x$ in $\partial M$, magnetic geodesics in $M$ have no self intersections. In particular,any $x\in M$ has only one preimage under $exp^\mu_x$.  We can then follow the same argument as for $x\in\partial M$ to show that $exp^\mu_x$ is a global diffeomorphism from $(exp^\mu_x)^{-1}(M)$ to $M$, for all $x \in M$.
\end{proof}


\section{Rigidity for Surfaces}\label{Rigidity for Surfaces}

Consider a magnetic field on a surface $\widehat{M}$ all of
whose orbits are closed, and consider a magnetically simple region $R$ on it.
We want to prove that there is no way of changing the metric and magnetic field
in this region in such a way that all orbits are still closed.

In the previous section we saw that such a region is magnetically rigid,
therefore it can't be changed on the region preserving the scattering data.
Here we will look at the general behavior of such a magnetic flow to ensure
that there are no other metrics with all its orbits closed. We want to rule out
the case where a magnetic geodesic that passes through the region, after coming
out at a different spot and following the corresponding orbit, goes back into
the region and exits at the exit point and direction of the original first
magnetic geodesic, therefore forming a closed orbit out of two (or more)
segments of the original orbits, like in figure \ref{double orbit}.

To show this, assume that we have such a magnetic field. Assume, moreover, that
the region $R$ is such that every magnetic geodesic passes through $R$ at most once. This condition restricts both the size of $R$ and the flow, since in a flow where the orbits have many self-intersections such a region might not exist. On the other hand, if all orbits are simple we can see by compactness that there are small regions with this property. \\

\begin{proof}[Proof of Theorem \ref{closed orbits implies scattering}]

Consider the unit tangent bundle $SM$, and the magnetic geodesic vector field
$G$ i.e. the vector field that generates the magnetic flow on the unit tangent
bundle. For the sake of simplicity of the exposition we will assume first that
$SM$ is oriented. Then ${SM}$ is a compact orientable $3$-dimensional manifold,
and ${G}$ is a smooth vector field that foliates ${SM}$ by circles. By a
theorem of Epstein \cite{Ep}, this foliation is $C^\infty$ diffeomorphic to a
Seifert fibration. In particular, any orbit has a neighborhood diffeomorphic to
a standard fibered torus.

Note that to each orbit on ${SM}$ we can uniquely associate a magnetic
geodesic, by projecting the orbit back to $M$. We will use this correspondence
freely. As a Seifert fibration, the  base $B$ or space of orbits of ${SM}$ is a
$2$-dimensional orbifold.

Let ${SR}$ be the subset of ${SM}$ that corresponds to the region $R$, and
$S\partial{R}$ the subset of $SM$ corresponding to vectors tangent to the
boundary of $R$. The orbit of a point in $S\partial{R}$ corresponds to a
magnetic geodesic that is tangent to $R$, and since $R$ is strictly
magnetically convex, it is tangent only at one point. This magnetic geodesic
corresponds exactly to one in $\widehat{M}$, and therefore stays away from $R$
thereafter. This means that each orbit contains at most one point of
$S\partial{R}$, so the set of orbits passing through it forms a
$1$-dimensional submanifold on $B$, we will denote it by $R_0$. \\

Let $m : B\to \mathbb{N}$ be a function that counts the number of times the
orbit pases through ${SR}$ in a common period. For regular orbits this is the
number of times it passes through ${SR}$.  If the orbit is singular it has a neighborhood
diffeomorphic to an $(a,b)$ torus, that is a torus obtained by gluing two
faces of a cylinder with a rotation by an angle of $2\pi b/a$. In this case the
common period is $a$ times the period of the singular orbit. Therefore, $m$
will be $a$ times the number of times the orbit passes through ${SR}$. Since
the other orbits in the neighborhood will be completed when the singular orbit
is traveled $a$ times, $m$ will be, in general, continuous at such points. In
fact, if a magnetic geodesic  intersects $\partial R$ transversally (or not at
all), we can chose a neighborhood small enough that all intersections are
transverse, and therefore $m$ will be constant. The only discontinuities occur
when a magnetic geodesic is tangent to $\partial R$, that is exactly at the
orbits in $R_0$.

We will now look at these discontinuities. If a magnetic geodesic $\gamma$ corresponds
to an orbit $b$ in $R_0$, it is tangent to $R$ at a point $\gamma(0)$. It agrees with a
magnetic geodesic in $\widehat{M}$, so it never reaches $R$ again and $m(b)=1$.
By the magnetic convexity of $M$, there is a $\delta$ small enough that each magnetic geodesic in the ball $B_\delta(\gamma(0))$ goes through $R$ at most once. Moreover, since the magnetic geodesic is compact, we can find an $\varepsilon$ neighborhood $N_\varepsilon(\gamma)$ such that it only intersects $R$ close to $\gamma(0)$, i.e. $N_\varepsilon(\gamma)\cap R =B_\delta(\gamma(0))\cap R$.

If the  orbit $b$ corresponding to $\gamma$ is regular, orbits in a small enough neighborhood will correspond to nearby magnetic geodesics, completely contained in $N_\varepsilon(\gamma)$. These magnetic geodesics will then intersect $R$ only inside  $B_\delta(\gamma(0))$, and therefore at most once. Thus, these orbits will have $m$ equal to  $0$ or $1$.\\

We have that $B$ is a $2$-dimensional orbifold, and $R_0$ is a continuous curve
on it. The function $m$ is constant on each
connected component of $B\setminus R_0$, and takes values $0$ or $1$. On
$R_0$, the function $m=1$, except maybe at isolated
singular orbits. Since on
regular orbits $m=1$, we
can say that magnetic geodesics go through $R$ at most once, except maybe at a
finite number of singular ones. Any singular magnetic geodesics that is tangent must go
through $R$ only once. If a singular magnetic geodesics cuts $\partial R$ transversely, we
know that $m=1$. But the corresponding orbit pases through ${SR}$ exactly $m/a$
times, so $a=m=1$ and the geodesic is not singular.\\

In the case where $SM$ is non orientable, consider instead
its orientable double cover $\widetilde{SM}$, and the associated vector field
$\widetilde{G}$. Then $\widetilde{SM}$ is a compact orientable $3$-dimensional
manifold, and $\widetilde{G}$ is a smooth vector field that foliates $\widetilde{SM}$
by circles. We can follow the same arguments with a few modifications.

The correspondence between orbit on $\widetilde{SM}$ and magnetic geodesics is not
a $1-1$ correspondence, a magnetic geodesic lifts either to an orbit that
covers it twice, or two disjoint orbits. Nonetheless, will use this
correspondence freely, keeping in mind this possible duplicity.

Let $\widetilde{SR}$ be the subset of $\widetilde{SM}$ that corresponds to the region
$R$, and $\widetilde{S\partial R}$ the subset of $\widetilde{SM}$ corresponding to
vectors tangent to the boundary of $R$. Let $\widetilde B$ be space of orbits of
$\widetilde{SM}$ and $\widetilde{R_0}$ the set of orbits passing through
$\widetilde{S\partial R}$. The counting function $m : \widetilde{B}\to \mathbb{N}$ can
then take value $2$, since an orbit that covers a magnetic geodesic twice will
pas through $\widetilde{SR}$ twice. In fact, when $m(b)\neq 0$, it will be $1$ if
the magnetic geodesic corresponds to two disjoint orbits, and $2$ when it
corresponds to an orbit that covers it twice.

Since on regular orbits $m=2$ only on orbits that cover a magnetic geodesic twice, we
can say that magnetic geodesics go through $R$ at most once, except maybe for a
finite number of singular ones. Any singular one that is tangent must go
through $R$ only once, by assumption. If a singular orbit cuts $\partial R$ transversely, we
know that $m$ is at most $2$. But the orbit pases through $\widetilde{SR}$ $m/a$
times, so if it is singular $a=m=2$ and the geodesic goes through $R$ only
once.\\

Every magnetic geodesic goes through $R$ at most once, and outside $R$ they
agree with the magnetic geodesics from $\widehat{M}$. For the magnetic geodesics to
close, they have to exit $R$ in the same place and direction, therefore
preserving the scattering data.
\end{proof}

If the region $R$ is simple, we can use this result together with theorem
\ref{simple implies simple} to get rigidity. For surfaces of constant curvature is easy to see that
any circular disk that is strictly smaller than one of the orbit circles is a
simple domain. Corollary \ref{rigidity for constant curvature} can be stated in a more precise way as the following theorem.

\begin{thm}
Let $M$ be a surface of constant curvature $K$, and $k>0$ big enough that all circles of curvature $k$ are simple.
Let $r_k$ be the radius of a circle of curvature $k$, and $0<r< r_k$ such that $r+r_k$ is smaller than the injectivity radius of $M$. Let $R$ be a compact
region contained in the interior of a disk of radius $r$. Then the region
$R$ can't be perturbed while keeping all the circles of curvature $k$ closed.
\end{thm}

Consider $M$ with the constant magnetic field that has circles of curvature $k$
as magnetic geodesics. If $R$ is contained in a disk $D$ of radius $r$ we
can consider any perturbation $\widetilde{R}$ of $R$ as a perturbation $\widetilde{D}$
of $D$. Since $r<r_k$ the disk $D$ is simple. Also, since $r+r_k$ is smaller than the injectivity radius of $M$, any circle of curvature $k$ will go through $D$ at most once.
We can then use theorem \ref{closed orbits implies
scattering} to show that $D$ and $\widetilde{D}$ have the same scattering data.
Since $D$ is simple, and they have the same scattering data, by theorem \ref{simple implies simple}
$\widetilde{D}$ is also simple. But in  \cite[Theorem $7.1$]{DPSU} N. Dairbekov, P.
Paternain, P. Stefanov and G. Uhlmann proved that two $2$-dimensional simple
magnetic systems with the same scattering data are gauge equivalent. \\

If $M$ is not compact, consider instead of $M$ a compact quotient that contains
all the magnetic geodesics that pass through $D$. This can be achieved since
all this magnetic geodesics are inside a disk of radius $4r$, where $r$ is the
radius of a circle of curvature $k$.


\begin{thebibliography}{9999}

\bibitem[AS]{AS} D.V. Anosov and Y.G. Sinai, {\em Certain smooth ergodic systems}, Uspekhi Mat. Nauk, {\bf 22:5} (1967), 107 -- 172; MR, { 37, \# 370}; Russian Math. Surveys, {\bf 22:5} (1967), 103 -- 167

\bibitem[Ar1]{A61} V.I. Arnold, {\em Some remarks on flows of line elements and frames}, Soviet Math. Dokl., {\bf 2} (1961), 562 -- 564.

\bibitem[Ar2]{A86} V. I. Arnol'd, {\em The first steps of symplectic topology}, RUSS MATH SURV,{\bf 41}(6) (1986),\\ 1--21.

\bibitem[Ar3]{A88} V.I. Arnold, {\em On some problems in symplectic topology}, in Topology and Geometry   Rochlin Seminar, O.Ya. Viro (Editor), Lect. Notes in Math., vol. 1346, Springer, 1988.

\bibitem[BCG]{BCG}G. Besson, G. Courtois and S. Gallot, {\em Entropies et rigidit\'es des espaces localement sym\'etriques de courbure strictment n\'egative},  Geom. Funct. Anal.  {\bf 5} (1995), 731 -- 799.

\bibitem[CMP]{CMP}G. Contreras, L. Macarini, and G.P. Paternain. {\em Periodic orbits for exact magnetic flows on surfaces}, Int. Math. Res. Not., {\bf 8} (2004), 361 -- 387.

\bibitem[Cr1]{Cr90} C. Croke, {\em Rigidity for surfaces of non-positive curvature}, Comm. Math. Helv. {\bf 65} (1990) no.1, 150-169.

\bibitem[Cr2]{Cr91} C. Croke {\em Rigidity and the distance between boundary points} J. Diff. Geom. {\bf 33} (1991),\\ 445 -- 464.

\bibitem[Cr3]{Cr04} C. B. Croke, {\em Rigidity theorems in Riemannian geometry}, Geometric methods in inverse problems and PDE control, 47 72, IMA Vol. Math. Appl., 137, Springer, New York, 2004.

\bibitem[DPSU]{DPSU} N. Dairbekov, P. Paternain, P. Stefanov and G. Uhlmann, {\em The boundary rigidity problem in the presence of a magnetic field}, Adv. Math.  {\bf 216} (2007), 535 -- 609.

\bibitem[Da]{Da} G. Darboux, {\em Le\c{c}ons sur la th\'eorie g\'en\'erale des surfaces et les applications g\'eom\'etriques du calcul infinit\'esimal}, Vol. 3, Gauthier-Villars, Paris, 1894.

\bibitem[Ep]{Ep} D.B.A. Epstein,{\em Periodic flows on 3-manifolds}, Ann. of Math. {\bf 95} (1972), 68 -- 82

\bibitem[Gi1]{Gi87} V.L. Ginzburg, {\em New generalizations of Poincar´e s geometric theorem}, Functional Anal. Appl., {\bf 21} (2) (1987), 100 -- 106.

\bibitem[Gi2]{Gi96} V.L. Ginzburg, {\em On closed trajectories of a charge in a magnetic field. An application of symplectic geometry}, in Contact and Symplectic Geometry (Cambridge, 1994), C.B. Thomas (Editor), Publ. Newton Inst., 8, Cambridge University Press, Cambridge, 1996, p. 131 -- 148.

\bibitem[Grog]{Gr99} S. Grognet, {\em Flots magn\'etiques en courbure n\'egative} (French), Ergodic Theory \& Dynam. Systems {\bf 19} (1999) no. 2, 413 -- 436.

\bibitem[Gro]{Gro}   M. Gromov, {\em Filling Riemannian manifolds}, J. Diff. Geom. {\bf 18} (1983), 1 -- 147.

\bibitem[Le]{Le} M. Levi, {\em On a problem by Arnold on periodic motions in magnetic fields}, Commun. Pure Appl. Math.{\bf 56} (2003) No.8, 1165 -- 1177.

\bibitem[Mi]{Mi} R. Michel, {\em Sur la rigidit\'e impos\'ee par la longuer des g\'eod\'esiques}, Inv. Math. {\bf 65} (1981), 71 -- 83.

\bibitem[Ni]{Ni} C. Niche, {\em On the topological entropy of an optical Hamiltonian flow}, Nonlinearity,{\bf 14} (2001), 817 -- 827.

\bibitem[No]{No} S.P. Novikov, {\em The Hamiltonian formalism and a many-valued analogue of Morse theory}, Russian Math. Surveys, {\bf37 }(5) (1982), 1 -- 56.

\bibitem[NT]{NT} S.P. Novikov and I.A. Taimanov, {\em Periodic extremals of many-valued or not everywhere positive functionals}, Sov. Math. Dokl., {\bf 29}(1) (1984), 18 -- 20.

\bibitem[Ot]{Ot} J.-P. Otal, {\em Le spectre marqu\'edes longueurs des surfaces \`e courbure n\'egative}, Ann. of Math. {\bf131} (1990), 151 -- 162.

\bibitem[Pa]{Pa}G.P. Paternain, {\em Geodesic flows}, Progress in Mathematics, 180 Birkauser 1999

\bibitem[PP]{PP97} G.P. Paternain and M. Paternain, {\em First derivative of topological entropy for Anosov geodesic flows in the presence of magnetic fields}, Nonlinearity,{\bf 10} (1997), 121 -- 131.

\bibitem[Sc]{Sc} M. Schmoll {\em On the asymptotic quadratic growth rate of saddle connections and periodic orbits on marked flat tori} Geom. Funct. Anal. {\bf 12} (2002), no. 3, 622 -- 649.

\bibitem[SS]{SS} B. Schmidt and J. Souto, {\em Chords, light, and another synthetic characterization of the round sphere}, arXiv:math.GT/0704.3642

\bibitem[Sch]{Sch} M. Schneider. {\em Closed magnetic geodesics on S2}, Preprint, arXiv:0808.4038 [math.DG], 2008.

\bibitem[Ta]{Ta} I.A. Taimanov, {\em Closed extremals on two-dimensional manifolds}, Russian Math. Surveys, {\bf 47}(2) (1992), 163 -- 211.

\bibitem[Zo]{Zo} O. Zoll {\em \"Uber Fl\"achen mit Scharen geschlossener geod\"atischer Linien}, Math. Ann. {\bf 57} (1903), 108 -- 133.

\end{thebibliography}
\end{document}